\documentclass[11pt,a4paper]{amsart}
\usepackage{amssymb}
\usepackage{amscd}
\usepackage{amsmath,amsfonts,stmaryrd,amsaddr}
\input xy 
\xyoption{all}
\usepackage{tikz}
\usetikzlibrary{arrows,chains,matrix,positioning,scopes,cd}

\theoremstyle{plain}
\newtheorem{thm}{Theorem}[section]
\newtheorem{lem}[thm]{Lemma}

\newtheorem{pro}[thm]{Proposition}
\newtheorem{cor}[thm]{Corollary}

\theoremstyle{definition}

\newtheorem{exa}[thm]{Example}
\newtheorem{rem}[thm]{Remark}

\DeclareMathOperator{\Hom}{Hom}

\DeclareMathOperator{\Aut}{Aut}
\DeclareMathOperator{\End}{End}

\DeclareMathOperator{\proj}{proj}

\DeclareMathOperator{\modu}{mod}

\newcommand{\sym}{sym}
\newcommand{\sm}{\text{-}}

\DeclareMathOperator{\id}{id}

\DeclareMathOperator{\pdim}{proj.dim}

\newcommand{\oTo}{\xymatrix{ \ar@{^{(}->}[r]|{\mathbf{O}}& }} 
\newcommand{\cTo}{\xymatrix{ \ar@{^{(}->}[r]|{\mathbf{|}}& }} 
\newcommand{\coTo}{\xymatrix{ \ar@{^{(}->}[r]|{\mathbf{O}}|{\mathbf{|}}& }} 

\DeclareMathOperator{\Bild}{Im}

\DeclareMathOperator{\kdual}{D}
\DeclareMathOperator{\Out}{Out}
\DeclareMathOperator{\Inn}{Inn}
\DeclareMathOperator{\DPic}{DPic}

\newcommand{\ep}{\varepsilon}

\newcommand{\si}{\sigma}

\newcommand{\la}{\lambda}
\newcommand{\La}{\Lambda}
\newcommand{\al}{\alpha}

\newcommand{\Z}{\mathbb{Z}}

\newcommand{\mcT}{\mathcal{T}}

\newcommand{\mcZ}{\mathcal{Z}}

\DeclareMathOperator{\add}{add}

\DeclareMathOperator{\tilt}{tilt}
\DeclareMathOperator{\silt}{silt}

\DeclareMathOperator{\Ka}{K}

\usepackage{todonotes}

\begin{document}

\title{On the tilting complexes for the Auslander algebra of the truncated polynomial ring}

\date{\today}
%
%
\author{Julia Sauter}
\address{Fakult\"at f\"ur Mathematik, Universit\"at Bielefeld,\\ Postfach 100 131, D-33501 Bielefeld, Germany,  \\ 
Phone : 0049 521 106 5036}
\email{jsauter@math.uni-bielefeld.de}

\begin{abstract} 
We give a bijection between the tilting complexes in the bounded homotopy category of the Auslander algebra of $\rm{K[T]/T^n}$ and $\mathbb{Z} \times \rm{B}_n$ where $\rm{B}_n$ is the Artin braid goup of type $A$ with $n-1$ generators. The tilting complexes have mutation components parametrized by $\mathbb{Z}$ and each component has a natural faithful and transitive operation of $\rm{B}_n$. This also implies that the derived Picard group of this algebra is isomorphic to the direct product of its outer isomorphism group and $\mathbb{Z} \times \rm{B}_n$. This work is to be seen as a continuation of the work of Geuenich and an application of the work of Aihara and Mizuno on tilting complexes of preprojective algebras of Dynkin type. 

\end{abstract}
\maketitle

\section{Introduction}
Preprojective algebras and Auslander algebras are two of the most intensely studied classes of algebras in the representation theory of algebras with geometric applications to Lie and cluster theory, e.g. \cite{GP}, \cite{AIRT}, \cite{AIR15}, \cite{BIRS}, \cite{GLS06},\cite{GLS}, \cite{KS},\cite{Lu},\cite{GLS07}, \cite{BHRR} etc. In this article we focus on the two special cases preprojective algebras of Dynkin type $A$ and Auslander algebras of truncated polynomial rings. It has already been observed that these algebras are related, see for example \cite{RZ} and \cite[Theorem 1.3]{IZ}. We will show that tilting complexes for the Auslander algebra have a \emph{symmetric double} embedding into the tilting complexes of the preprojective algebra (see section \ref{symm-double}). Using this we can classify all tilting complexes for the Auslander algebra of the truncated polynomial ring.     
Let us introduce the algebras we are studying and then state our results. Let $n\geq 2$ be a natural number: 
\begin{itemize}
\item[(1)]
Let $\Pi_n :=\Pi (\mathbb{A}_n)$ be the \textbf{preprojective algebra of Dynkin type $A_n$}. It is given by the bound quiver with relations 
\[ 
\xymatrix{ 
1 \ar@<.3em>[r]^{\al_1} & 2 \ar@<.3em>[l]^{\beta_1}\ar@<.3em>[r]^{\al_2} & \ar@<.3em>[l]^{\beta_{2}}\cdots  \ar@<.3em>[r]^{\al_{n-2}} & n \! - \! 1 \ar@<.3em>[r]^{\al_{n-1}} \ar@<.3em>[l]^{\beta_{n-2}}& n \ar@<.3em>[l]^{\beta_{n-1}},
} \quad \quad 
\begin{aligned}
&\beta_1\al_1, \\
&\al_i\beta_i-\beta_{i+1}\al_{i+1}, \;  1 \leq i \leq n-2, \\
&\al_{n-1}\beta_{n-1}\end{aligned}
\]
We denote by $e_1, \ldots , e_{n}\in \Pi_n$ the complete set of primitive orthogonal idempotents corresponding to the vertices in the quiver and by $P_i := e_i \Pi_n$ the indecomposable projective right $\Pi_n$-module.    
\item[(2)] Let $\La_n =\rm{Aus} (K[T]/ (T^n)):= \End_{K[T]} (\bigoplus_{i=1}^n K[T]/(T^i))$ be the \textbf{Auslander algebra of $\rm{K [T]/T^n}$}. It is given by the bound quiver with relations 
\[ 
\xymatrix{ 
1 \ar@<.3em>[r]^{\al_1} & 2 \ar@<.3em>[l]^{\beta_1}\ar@<.3em>[r]^{\al_2} & \ar@<.3em>[l]^{\beta_{2}}\cdots  \ar@<.3em>[r]^{\al_{n-2}} & n \! - \! 1 \ar@<.3em>[r]^{\al_{n-1}} \ar@<.3em>[l]^{\beta_{n-2}}& n \ar@<.3em>[l]^{\beta_{n-1}},
} \quad \quad 
\begin{aligned}
&\beta_1\al_1, \\
&\al_i\beta_i-\beta_{i+1}\al_{i+1}, \;  1 \leq i \leq n-2
\end{aligned}
\]
We denote by $\ep_1, \ldots , \ep_n$ the complete set of primitive orthogonal idempotents corresponding to the vertices in the quiver and by $Q_i := e_i \La_n$ the indecomposable projective right $\La_n$-module.     
\end{itemize}

Furthermore, we define $\Pi_1:= \La_1=K$. Both families of algebras are well studied and have very different properties: 
\begin{itemize}
\item[(1)] 
The algebras $\Pi_n$ are finite-dimensional self-injective, their stable categories are equivalent to the cluster category of Dynkin type $A_n$. 
Furthermore, support $\tau$-tilting modules, $2$-term-tilting complexes and tilting complexes have been studied by Mizuno \cite{Mi} and Aihara-Mizuno \cite{AM}. 
\item[(2)]
The algebras $\La_n$ are finite-dimensional of global dimension and dominant dimension $2$ if $n\geq 2$. They have a unique quasi-hereditary structure which is induced from the inclusion of the projectives $Q_1\subset Q_2 \subset \cdots \subset Q_n$. 
We have $\La_n / (\ep_n)\cong \Pi_{n-1}$ and the associated recollement has been studied by \cite{RZ}. \\
Furthermore, in the work \cite{IZ} the support $\tau $-tilting modules and classical tilting modules are studied (and shown to be in bijection to support $\tau$-tilting modules of $\Pi_n$). In recent work of \cite{G} all tilting modules are studied. It is surprising that it has only finitely many. In  \cite{HP},  exceptional sequences in the module category are classified. 
\end{itemize}

Tilting and support $\tau$-tilting theory are one of the main research areas in representation theory of finite-dimensional algebras. The intriguing observation is that both families of algebras share very similar combinatorial parametrizations of tilting related modules or complexes (support-$\tau$-tilting modules, tilting modules, tilting complexes etc.) by Weyl groups or braid groups.  

For any finite-dimensional algebra $\La$ we denote by $\Ka^b(\La\!-\!\proj)$ the homotopy category of 
bounded finite-dimensional projective right $\La$-modules and by $\tilt \La$ (resp. $\silt \La$) the set of all isomorphism classes of basic tilting (resp. silting) complexes in $\Ka^b(\La\!-\!\proj)$. 
Furthermore, if $P$ is a basic presilting complex in $\Ka^b(\La\!-\!\proj)$, then we write $\tilt_P \La$ (resp. $\silt_P\La$) for the subset of tilting (resp. silting) complexes in $\tilt \La$ (resp. $\silt \La$) having $P$ as a summand. 

Recall that the artin braid group $B_n$ is generated by $s_1, \ldots , s_{n-1}$ subject to the braid relations (see for example \cite{KT}). We denote by $B_n^+$ the submonoid generated by $s_1, \ldots , s_{n-1}$. For $x,y \in B_n$ we define the right divisibility order $y \geq_L x $ by 
$yx^{-1}\in B_n^+$.

We show the following result for $\La_n$:
\begin{thm}  \label{main-thm}(cf. section \ref{main-step})
We have $\tilt \La_n = \bigsqcup_{m\in \mathbb{Z}} \silt_{Q_n[m]} \La_n$ and the subsets $\silt_{Q_n[m]} \La_n $ are precisely the mutation connected components. Furthermore, we have an explicit isomorphism of posets 
\[ \rho \colon B_n \to \silt_{Q_n[m]} \La_n, \quad x \mapsto T_x \]
where we consider $B_n$ with the right divisibility order $\geq_L$.
\end{thm}

As a corollary we show that $\La_n$ has $2n!$ two-term tilting complexes, see Corollary \ref{two-term}. 
Furthermore, for $\La_n$ and $\Pi_n$ their derived equivalence class coincides with their isomorphism class as algebras which is a very rare and strong \emph{rigidity} property, cf. Corollary \ref{rigidity}. In fact, this is explained using a factorization $\overline{\rho}\colon B_n \to \DPic (\La_n), x \mapsto \mathbb{T}_x$ of the map $\rho$ over the derived Picard group (see \cite[Prop. 6.3]{G} or section \ref{braid-DPic}).  

\begin{thm} ($=$ Prop. \ref{derived-Pic} $+$Prop \ref{out-La})\label{Thm-DPic} The following map 
\[
\begin{aligned} \Out (\La_n)\times \mathbb{Z} \times B_n &\to \DPic ( \La_n)  \\
(\phi , m , x) &\mapsto ({}_{\phi^{-1}}\La_n \otimes_{\La_n}^{L}\mathbb{T}_x)[m]
\end{aligned}
\]
is an isomorphism of groups. Furthermore, we have $\Out (\La_n)$ is isomorphic to $\Aut_{K-alg} (K[T]/(T^n))$. 
\end{thm}

\section{On the symmetric doubling of tilting complexes} \label{symm-double}

Let $\Pi$ be a basic finite-dimensional self-injective algebra with Nakayama automorphism $\nu_{\Pi}$ of order two, i.e. $\nu_{\Pi}^2=\id$. This induces an involution $\nu$ on the set of isomorphism classes of indecomposable projective $\Pi$-modules.
Let $e \in \Pi$ be an idempotent such that $e\Pi $ is a direct sum of a complete set of representatives for the $\nu$-orbits in the set of indecomposable projective (right) $\Pi$-modules. 
Let $\Pi^{\nu}$ be the maximal basic summandwise $\nu$-invariant projective (right) $\Pi$-module, so by assumption $\Pi^{\nu} \in \add e\Pi $ (but it is possible that $\Pi^{\nu}=0$).  \\
We set $\La := e\Pi e$. 
From the assumption on $e$ we conclude $\lvert \Pi \rvert + \lvert \Pi^{\nu} \rvert = 2\lvert e\Pi  \rvert =  2 \lvert \La \rvert$. 

By $\La\!-\!\modu$ and $\Pi\!-\!\modu$ we denote the categories of finite-dimensional right $\La$- and $\Pi$-modules. 
The functors $\ell = - \otimes_{\La} e\Pi , r=\Hom_{\La}(\Pi e ,-) \colon \La\!-\!\modu \to \Pi\!-\!\modu$ are fully faithful, $\ell$ preserves projectives and $r$ preserves injectives. They are left and right adjoint functors of the restriction functor $e$ which maps $X \mapsto Xe$.

That means we get two induced functors (which we denote by the same letter) 
\[
\ell \colon \Ka^b(\La\!-\!\proj )\to \Ka^b (\Pi\!-\!\proj), \quad r\colon \Ka^b(\La \!-\!{\rm inj} )\to \Ka^b (\Pi\!-\!\proj)   \]
Both functors are fully faithful, compatible with shift and preserve triangles. 
Furthermore, we have the Nakayama functors defined componentwise on the complexes, denoted by $\nu_{\Pi}$ and $\nu_{\La}$. 
For every idempotent $e\in \Pi$ and $\La, r, \ell$ defined as before we observe $r \nu_{\La}(\La) = r(\kdual \La)= \kdual (\Pi e) = \nu_{\Pi} (e\Pi ) = \nu_{\Pi} \ell (\La)$ and this implies the identity 
\[  \nu_{\Pi} \circ \ell = r \circ \nu_{\La} \colon \Ka^b(\La \!-\!\proj )\to \Ka^b (\Pi \!-\!\proj) . \]
By definition of $e$ we conclude 
$\add \ell (\La) \cap \add r(\kdual \La) = \add (e\Pi)  \cap \add (\nu_{\Pi} e\Pi ) = \add (\Pi^{\nu}) $, this also implies $\Pi^{\nu}= re (\Pi^{\nu}) = \ell e (\Pi^{\nu})$, and we conclude $\add (\Pi) =\add (e\Pi  \oplus \nu_{\Pi} (e\Pi )) = \add (\ell (\La) \oplus r(\kdual \La))$, in particular: 
\[ \Pi \cong \ell(\La ) \oplus \left(r(\kdual \La )  / \Pi^{\nu} \right) \] 
where $(-)/\Pi^{\nu}$ means that we take a complement of the summand $\Pi^{\nu}$.  
Let $Q=\Pi^{\nu}e$, then $Q$ is a projective-injective $\La$-module with 
$\ell (Q) =\Pi^{\nu} = r(Q)$ implying it is summandwise $\nu_{\La}$-invariant (i.e. top and socle are the same for every indecomposable summand of $Q$). \\

We also define 

\[\tilt^{e-\sym} :=\{ T\in \tilt_{\Pi^{\nu}} \Pi  \mid \add (T)= \add (\ell (P)\oplus \nu \ell (P)) \text{ for some }P\in \tilt \La \} \] 

\begin{lem} \label{sym}
The map $\Phi\colon \tilt \La  \to \tilt \Pi $ mapping $P$ to a basic module $T$ with $\add (T) = \add (\ell(P) \oplus \nu \ell(P))$ is compatible with mutation of tilting complexes. It restricts to a bijection (still denoted by $\Phi$)
\[ 
\Phi \colon \tilt_Q \La \to \tilt^{e-\sym}
\]
\end{lem}

\begin{proof}
\begin{itemize}
\item[(*)] We first prove that $\Phi$ is well-defined. 
Let $i \in\Z, i\neq 0$. In this proof we use the following shorthand $(X,Y):= \Hom_{K^b} (X,Y)$ and $\nu=\nu_{\Pi}$. Using that $\ell , \nu\ell$ are fully faithful we have that $(\ell(P)\oplus \nu \ell (P), \ell(P) [i]\oplus \nu \ell (P) [i] ) $ equals 
\[(P,P[i])\; \; \oplus \;\; (\ell(P) , \nu \ell (P) [i]) \; \; \oplus \; \; (\nu \ell(P) , \ell (P) [i])\; \;  \oplus \; \; (P,P[i])\]
Since $P$ is a tilting complex, we conclude $(P,P[i]) =0$. Since $\nu$ is a Serre functor  on $ \Ka^b (\Pi\!-\!\proj)$ and $P$ tilting we conclude 
\[
(\ell(P) , \nu \ell (P) [i]) \cong \kdual ( \ell (P)[i], \ell (P) ) =\kdual (P,P[-i]) =0
\]
Since $\nu^2 \cong \id$ we have  $(\nu \ell(P) , \ell (P) [i])  \cong ( \ell (P), \nu \ell (P)[i]) =0$. \\
Next we want to see that $\Pi $ is in the thick subcategory generated by $\ell (P) \oplus \nu \ell (P)$.  
Since $\ell $ maps triangles to triangles and $\La $ is in the thick subcategory generated by $P$ we have that: 
$\ell (\La) $ is in the thick subcategory generated by $\ell (P)$. Furthermore $r (\kdual \La) = r (\nu \La ) = \nu \ell (\La ) $ is in the thick subcategory generated by $\nu \ell (P)$. This implies $\Pi$ which is a summand of $\ell (\La) \oplus r (\kdual \La) $ is in the thick subcategory generated by $\ell (P) \oplus \nu \ell (P)$. 
\item[(*)] We now explain that it is compatible with mutation of summands: 
Assume $X\oplus U, Y\oplus U$ are tilting complexes in $\Ka^b(\La\!-\!\proj )$, $X,Y$ are summands and we have a triangle $X\xrightarrow{f} U_0\to Y \to X[1]$ with $f$ a minimal left $\add(U)$-approximation. The functors  
$\ell, \nu \ell$ map triangles to triangles. Taking the direct sum we get a triangle 
\[
\ell(X) \oplus \nu \ell (X) \xrightarrow{g} \ell(U_0) \oplus \nu \ell (U_0)\to \ell(Y) \oplus \nu \ell (Y) \to \ell(X)[1] \oplus \nu \ell (X)[1]
\]
We would like to see that $g$ is a left $\add(\ell(U) \oplus \nu \ell (U))$-approximation. 
So we apply $(-, \ell(U) \oplus \nu \ell (U))$ to the triangle and get an exact sequence of vector spaces - to fit into one line we  here we define $L:=\ell \oplus \nu \ell $ 
\[ (L (Y), L (U)) \to (L (U_0) , L (U)) \to (L (X), L (U)) \to (L (Y)[-1], L (U)) \]
Now, the last term equals 
\[ (\ell\oplus \nu \ell (Y), \ell\oplus \nu \ell (U)[1]) = (Y, U[1]) \oplus (\ell (Y), \nu \ell (U)[1]) \oplus (\nu \ell (Y), \ell (U)[1]) \oplus (Y, U[1]) \]
Since $U\oplus Y$ is silting we have $(Y, U[1])=0$. Since it is even tilting by assumption  we get $(\ell (Y), \nu \ell (U)[1]) = \kdual ( U[1], Y) = \kdual (U, Y[-1])=0$ and 
$(\nu \ell (Y), \ell (U)[1]) = (\ell (Y) , \nu \ell (U) [1] ) = 0$. This proves that $g$ is a left $\add(\ell(U) \oplus \nu \ell (U))$-approximation.
\item[(*)] $\tilt^{e-sym}$ is the image of the restriction of $\Phi$ by definition. It is clear that the restriction is injective because for $P \in \tilt_Q \La$ we have that $T:=\ell \oplus \nu \ell (P)$ fulfills $\add \ell (P) = \add (T) \cap \Bild \ell $. 
\end{itemize}
\end{proof}

\begin{exa}
Let $\Pi$ be a preprojective algebra of Dynkin type, the Nakayama automorphism has order $2$, except for ${\rm char} (K)=2$ and Dynkin type $\mathbb{D}_{2n} (n \geq 2 )$, $\mathbb{E}_7$ or $\mathbb{E}_8$, cf. \cite{BBK}. But for a field of arbitrary characteristic and all of the previous Dynkin types the preprojective algebra is weakly symmetric, i.e. its Nakayama permutation ( $=$ the automorphism on the Grothendieck group) is the identity. 
For weakly symmetric algebras, the previously constructed algebra $\La$ equals $\Pi$. So nontrivial examples of the previous map $\Phi$ are the remaining Dynkin types $\mathbb{A}_n (n\geq 2),\mathbb{D}_{2n+1} (n \geq 2 )$ or $\mathbb{E}_6$. 
\end{exa}

\subsection{In the special case of the preprojective algebra of type $\mathbb{A}_{2n-1}$}
Let $e=e_J \in \Pi_{2n-1}$ be of the form $e= e_n + \sum_{i\in J} e_i + \sum_{i\in J^c} e_{2n-i} $ for some subset $J \subset \{1, \ldots , n-1\}$ (by varying the $J$ we obtain all $2^{n-1}$ possible choices of idempotents $e \in \Pi=\Pi_{2n-1}$ as considered in the previous section). Set $\La=e\Pi_{2n-1}e$, for $e= \sum_{i=1}^ne_i\in \Pi_{2n-1}$ we observe $\La\cong \La_n$. 
We fix primitive orthogonal idempotents for the algebra $\La$:  $\epsilon_n$, $\epsilon_i, i\in J$, $\epsilon_{2n-i}, i \in J^c=\{1, \ldots, n-1\}\setminus J$ such that $P_ie=e\Pi_{2n-1} e_i =  \epsilon_i\La=: Q_i$ for $i\in \{n\} \cup J \cup (2n-J^c)$. 

Recall the map 
\[ 
\Phi =\Phi_J\colon \tilt_{Q_n} \La \to \tilt_{P_n} \Pi_{2n-1}
\]
maps a tilting complex $P$ to a basic complex $T$ with $\add (T) = \add (\ell(P) \oplus \nu \ell(P))$.  We have the following observation in this situation (using that the module $Q=Q_n$ is indecomposable in this case). 
\begin{lem} \label{last-part-of-thm}
Every $P$ in $\tilt \La $ has a summand of the form $Q_n[j]$ for some $j \in \Z$. 
\end{lem}

\begin{proof}
Let $P$ be in $\tilt \Ka^b (\La \!-\!\proj)$. Since $\ell (P)\oplus \nu \ell (P)$ has $2n$ summands there exist indecomposable summands $X,Y$ of $P$ such that $\ell (X) \cong \nu \ell(Y)= r\nu (Y)$. By definition we have $\Bild \ell = \Ka^b (\add (P_n \oplus \bigoplus_{i\in J} P_i\oplus \bigoplus_{i\in J^c}P_{2n-i}))$ and $\Bild r = \Ka ^b (\add(P_n \oplus \bigoplus_{i\in J} P_{2n-i}\oplus \bigoplus_{i\in J^c}P_{i}))$ and it follows $\Bild \ell \cap \Bild r = \Ka^b(\add P_n)$. This implies that $X$ is in $\Ka^b(\add Q_n)$. 
Let $[b,a]$ be the interval where the cohomology of $X$ is non-zero. 
We get that we have an indecomposable complex $X$ of $\La $-modules of the form 
\[ 
0\to X_b=Q_n^k \xrightarrow{g} \cdots \xrightarrow{f}  X_a= Q_n^m \to 0
\]
with all non-zero maps are in the radical and $m,k>0$.
Since $Q_n$ is projective the map $f$ is not surjective and 
there is a non-zero map $t\colon X_a \to S_n=\rm{top}(Q_n)$ such that $tf=0$. Now we compose this with the inclusion into the socle of a summand of $X_b$. 
The composition and shift gives a not zero-homotop map $X\to S_n[-a]=S_n[-b+(b-a)] \to X[b-a]$ (it is not zero-homotopic since zero is the only possible homotopy).
Since $X$ is summand of a tilting complex, we conclude that $b=a$ and $X$ is a stalk complex. This implies $X=Q_n[m]$ for some $m \in \mathbb{Z}$.

\end{proof}

\section{Braid group operations and derived Picard groups} \label{braid-DPic}

In this section we will only look at $e= \sum_{i=1}^n e_i\in \Pi= \Pi_{2n-1}$, $\La = e \Pi e =\La_n$. \\
\noindent
\emph{Introduction of the braid groups:} 
Let $B_{2n}$ be the group generated by $s_1, \ldots , s_{2n-1}$ with relations $s_is_j=s_js_i$ whenever $\lvert i-j \rvert \geq 2$, $i,j \in \{ 1,2, \ldots , 2n-1\}$ and 
$s_is_{i+1}s_i =s_{i+1} s_i s_{i+1}$, $1 \leq i \leq 2n-2$. This is called the Artin braid group of Dynkin type $\mathbb{A}_{2n-1}$ cf. \cite{KT}. 
Let $B_{2n}^+\subset B_{2n}$ be the submonoid generated by $s_1, \ldots , s_{2n-1}$. Right divisibility is the partial order $\geq_L$ on $B_{2n}$ defined by $y \geq_L x$ whenever $yx^{-1} \in B_{2n}^+$.

Let $B_{\Delta^f}\subset B_{2n}$ be the $\nu$-folded braid group associated to the involution $\nu $ on $\{1,2, \ldots , 2n-1\}$ defined as $\nu(i)=2n-i$, i.e. it is the subgroup generated by $t_1, \ldots , t_{n-1} , t_n$ with $t_n=s_n$ and $t_i= s_i s_{2n-i}$ for $i \neq n$. We observe that the subgroup generated by $t_1, \ldots , t_{n-1}$ is isomorphic to an Artin braid group $B_n$, we will just denote it by $B_n$.\\
\noindent
\emph{Introduction of the derived Picard group:} Let $\La$ be a finite-dimensional $K$-algebra and $\La^e=\La^{op} \otimes_K \La $. We call $\mathbb{T} \in D^b(\La^e\!-\!\modu )$ a two-sided tilting complex if there exists an $\mathbb{S} \in D^b(\La^e\!-\!\modu )$ such that 
$\mathbb{T}\otimes^{L} \mathbb{S} \cong \mathbb{S}\otimes^{L} \mathbb{T} \cong \La $ in $D^b(\La^e\!-\!\modu )$. The set of isomorphism classes of basic two-sided tilting complexes is known as the \emph{derived Picard group} and denoted by $\DPic (\La)$. In \cite{Y}, it is shown that $\DPic (\La )$ is a group with respect to $-\otimes^{L}-$ with unit $\La$. A result of Rickard \cite{R}, states that for $\mathbb{T},\mathbb{S}$ as before, one has quasi-inverse triangle autoequivalences $-\otimes^{L} \mathbb{T}, - \otimes^L\mathbb{S}$ of $D^b(\La\!-\!\modu )$ and $\mathbb{T}\otimes^{L} -, \mathbb{S} \otimes^L-$ of $D^b(\La^{op}\!-\!\modu )$. 
In particular, if $\La$ has finite (right) global dimension by forgetting the left module structure, we have a forgetful map $F\colon \DPic (\La )\to \tilt \La$. 
Recall, the \emph{outer automorphism group} $\Out (\La) = (\Aut_{K-alg} (\La)) / {\rm Inn} (\La)$ where ${\rm Inn} (\La)$ is the subgroup of inner automorphism, i.e. $K$-algebra automorphisms of the form $x \mapsto \la x \la^{-1}$ for some $\la \in \La^*$.  Furthermore for $\mathbb{T},\mathbb{T}' \in \DPic (\La) $ we have $\mathbb{T}_{\La}\cong \mathbb{T}'_{\La}$ in $D^b(\La\!-\!\modu)$ if and only if there is a $\phi \in \Out (\La )$ such that $\mathbb{T}\cong {}_{\phi} \La \otimes^L \mathbb{T}'$ in $D^b(\La^e\!-\!\modu )$, cf. \cite[Prop. 2.3]{RZ03}.

\begin{itemize}
    \item[(1)] for $\Pi_{2n-1}$: Aihara's and Mizuno's bijection \cite[Thm 6.6]{AM} 
\[ B_{\Delta^f} \to \tilt \Pi_{2n-1}\]
In \cite{Mi17}, Mizuno found an isomorphism of groups $\Out (\Pi_{2n-1}) \ltimes B_{\Delta^f} \to \DPic (\Pi_{2n-1})$. 
In \cite[Prop. 6.2.2]{I}, it was shown $\Out (\Pi_{2n-1})$ is isomorphic to $\Aut_{K-alg} (K[T] /(T^n))\ltimes \langle \iota  \rangle $ where $\iota$ is the involution defined by $\iota (e_i)=\iota (e_{2n-i}), \iota (\alpha_i )= \beta_{2n-1-i}$ and $\iota (\beta_i)= \alpha_{2n-1-i}$.  

    \item[(2)] for $\La_n$: In \cite[Prop. 7.1]{G} a poset morphism 
    \[ \rho \colon (B_n, \geq_{L}) \to \tilt_{Q_n} \La_n, x \mapsto T_x \]
 is constructed. Furthermore it is shown to factorize into a group homomorphism and the forgetful map $F\colon \DPic (\La_n) \to \tilt \La_n$, i.e. there is a group homomorphism $\overline{\rho}\colon B_n \to \DPic (\La_n), x \mapsto \mathbb{T}_x $ such that $\rho = F \circ \overline{\rho}$, i.e. $(\mathbb{T}_x)_{\La_n}\cong T_x$, cf. \cite[Prop. 6.3]{G}. 
  In loc. cit. it is also explained that there exists even a braid group action of $B_n$ on $\tilt \La_n$, the previous map is just the action/multiplication map on $\La_n$.  
\end{itemize}


Now, observe that Aihara's and Mizuno's bijection restricts to an injective map 
\[
\mu \colon B_n \to \tilt_{P_n} \Pi_{2n-1} 
\]

\begin{rem}\label{trivial-rem}
By the definition of the braid group operations we have a commutative diagram 
\[
\xymatrix{ 
& B_n \ar[dl]_{\rho} \ar[dr]^{\mu}& \\
\tilt_{Q_n} \La_n \ar[rr]^{\Phi} && \tilt_{P_n} \Pi_{2n-1} }
\]
In particular, the map $\rho $ is injective. 
\end{rem}

Recall, for a finite-dimensional algebra $\La$ we have a partial order on $\silt \La$ given by 
$S \leq T$ if and only if $\Hom_{K^b(\La\!-\!\proj)}(T, S[i])=0$ for all $i>0$. For any $S\leq T$ in $\silt \La$ we define the interval $[S,T] := \{ X \in \silt \La \mid S \leq X \leq T \}$ seen as a poset with the partial order $\leq $ restricted from  $\silt \La $. 

\begin{lem}\label{first-obs}
Every $T \in \Bild \rho $ fulfills $\End_{\Ka^b( \La_n\! -\!\proj)} (T)^{op}\cong \La_n$ and we have an isomorphism of posets  
\[ 
[T[1],T] \longleftrightarrow [\La_n[1], \La_n]
\]
\end{lem}

\begin{proof}
By \cite[Prop. 6.3]{G}, we have that for every $T\in\tilt \La_n$, there exists a $\mathbb{T}\in \DPic (\La_n)$ with $\mathbb{T}_{\La_n}=T$. Let $\mathbb{S}$ such that $\mathbb{T} \otimes^L \mathbb{S}\cong \La_n$. 
By \cite{R} we have that $\End_{\La_n}(T)\cong \End_{D^b(\La_n\!-\!\modu )} (\mathbb{T}_{\La_n}) \cong \End_{D^b(\La_n\!-\!\modu )} (\mathbb{T}\otimes^L \mathbb{S}) \cong \End_{\La_n} (\La_n)$. Since equivalences of triangle categories induce isomorphisms on the silting posets and commute with the shift, the second claim also follows from loc. cit. 
\end{proof}

\section{The main steps of the proofs}
\label{main-step}

We prove the following lemma which is the main step for completing the proof of Theorem \ref{main-thm}. 
We call a triangulated category $\mcT$ strongly silting-connected if all silting objects $T,S$ with $\Hom_{\mcT} (T, S[>0])=0$ (i.e. $S \leq T$), the silting object $S$ can be obtained from $T$ by iterated irreducible left mutation (cf.  \cite[Def. 3.1]{A}). 

\begin{lem} \label{main-obs}
Every $T\in \silt_{Q_n} \La_n$ is obtained from $\La_n$ by iterated irreducible silting mutations in $\silt \La$ such that every intermediate mutation has $Q_n$ as a summand.  
\end{lem}

\begin{proof}
We consider $\mcT:= \Ka^b (\La_n \!-\!\proj)/ \rm{thick} (Q_n )$ and the Verdier quotient functor $\pi \colon 
\Ka^b (\La_n \!-\!\proj) \to \mcT$. 
By the silting reduction theorem \cite[Cor. 3.8]{IY}, there is an isomorphism of posets 
\[
\begin{aligned}
R\colon \silt_{Q_n} \La_n &\to \silt \mcT \\
X & \mapsto \pi (X)=: \overline{X}
\end{aligned}
\]
We claim that for every $S\in \silt_{Q_n} (\La_n)$ the map \[ 
\begin{aligned}
\left[ S[1], S\right] \cap \silt_{Q_n}\La_n &\to [\overline{S}[1], \overline{S}]\\
X &\mapsto \overline{X}
\end{aligned}
\]
is surjective. To prove this let $S' \in \silt_{Q_n} \La_n$ such that $\overline{S'}= \overline{S}[1]$. 
By \cite[Lem 3.4]{IY}, since $S,S'\in \mcZ:= {}^{\perp}Q_n[>0]\cap Q_n[<0]^{\perp}$ we have that for every $\ell >0$ the natural maps 
\[ \Hom_{\Ka^b (\La_n \!-\!\proj)} (S', S[\ell ]) \to \Hom_{\mcT} (\overline{S}[1], \overline{S}[\ell ]) \text{ and } \Hom_{\Ka^b (\La_n \!-\!\proj)} (S , S'[\ell ]) \to \Hom_{\mcT} (\overline{S}, \overline{S}[\ell +1 ])\]
are isomorphisms. In particular we conclude $\Hom (S', S[>1])=0= \Hom (S, S'[>0])$ and $S[1] \leq S'\leq S$. This proves that the map is surjective since the composition $ \left[ S', S \right] \cap \silt_{Q_n}\La_n \subset \left[ S[1], S\right] \cap \silt_{Q_n}\La_n \to [\overline{S}[1], \overline{S}] $ is bijective. 

By \cite[Thm 2.4]{AM}, the category $\mcT$ is silting-discrete if and only if for a silting object $A$ and every silting object $B$ obtained from $A$ by iterated irreducible left mutation the interval $[B[1], B]=\{ T \in \silt \mcT \mid B \geq T \geq B[1] \}$ is finite. 
We consider the silting object $B_0=\overline{\La_n}$ and $B$ obtained from it by iterated irreducible left mutation, i.e. there exist irreducible left mutations $B_{i+1}$ of $B_i$ such that \[ B_0 > B_1 > \cdots > B_m=B .\]
The existence of an irreducible left mutation from $B_i$ to $B_{i+1}$ is by \cite[Thm 2.35]{AI} equivalent to that $B_i >B_{i+1}$ and the interval $[B_{i+1}, B_i]= \{ B_i, B_{i+1}\}$. 
Since $R$ is an isomorphism of posets we find $T_i \in \silt_{Q_n} \La_n$ with $\overline{T_i}= B_i$ and $T_0 > \cdots >T_m $ and $[T_{i+1}, T_i]\cap \silt_{Q_n} \La_n=\{ T_i, T_{i+1}\}$. We claim that 
$[T_{i+1}, T_i]=\{ T_i, T_{i+1}\}$, so assume $T_{i+1}\leq S \leq T_i$ in $\silt \La_n$. Then we have $\Hom_{\Ka^b (\La_n \!-\!\proj)} (S, Q_n[>0])=0=\Hom_{\Ka^b (\La_n \!-\!\proj)} (Q_n[<0], S)$ and therefore $S\oplus Q_n $ is a silting complex, implying $Q_n \in \add (S)$ and $S\in\{T_i, T_{i+1}\}$. \\
This implies that $T_{i+1}$ is an irreducible left mutation of $T_{i}$. In particular $T_m$ is obtained from $T_0=\La_n$ by iterated irreducible left mutation in $\silt_{Q_n}\La_n$. This implies that $T_i \in \Bild \rho \subset \tilt_{Q_n} \La_n$ and from Lemma \ref{first-obs} it follows that the interval $[T_m[1], T_m]$ is finite since $\La_n$ is $\tau$-tilting finite by \cite{IZ}. This implies that $[B[1], B]$ is also finite and therefore $\mcT $ is silting-discrete. 
By \cite[Cor. 3.9]{A}, if $\mcT$ is silting-discrete, then it is strongly silting connected. 
By the silting reduction theory this implies that $\silt_{Q_n} \La_n$ is strongly silting connected. 
\end{proof}

\begin{cor} \label{main-cor}
We have $\tilt_{Q_n} \La_n= \silt_{Q_n} \La_n$ and the map $\rho \colon B_n \to \tilt_{Q_n} \La_n $ is bijective. 
\end{cor}

\begin{proof}
Consider the composition $\rho \colon B_n \to \tilt_{Q_n} \La_n \subset \silt_{Q_n}\La_n$. It is injective (cf. remark \ref{trivial-rem}) and surjective by lemma \ref{main-obs}. But since the image of $\rho $ lies in $\tilt \La_n$ the claim follows. 
\end{proof}

The last step of the proof of Theorem \ref{main-thm} follows from Lemma \ref{last-part-of-thm}, which implies 
$\tilt \La_n = \bigsqcup_{m \in \Z} \tilt_{Q_n[m]}\La_n$. 

As a corollary of the main theorem, we obtain:
\begin{cor} \label{rigidity}
For every $T\in \tilt \La_n$ we have an isomorphism of algebras  $\End_{\Ka^b(\La_n-\proj)} (T)^{op} \cong \La_n$.  \end{cor}

\begin{proof}
This follows directly from the previous Corollary \ref{main-cor} and Lemma \ref{first-obs}. 
\end{proof}

We split Theorem \ref{Thm-DPic} into the following two propositions.

\begin{pro}\label{out-La}
The group ${\rm Out}(\La_n) $ is isomorphic to $\Aut_{K-alg} (K[T]/(T^n))$
\end{pro}

The proof is completely analogue to the proof of \cite[Prop. 6.2.2]{I}. 
\begin{proof}
Let $H=\{g \in \Aut_{K-alg}(\La_n) \mid g (e_i)=e_i, g (\alpha_j ) =\alpha_j\}$. 
1. We first show that $\Aut_{K-alg}(\La_n)$ is generated by $H$ and $\Inn (\La_n)$. 
Let $g \in \Aut_{K-alg}$, then there is an  $\la \in \La_n^*$ and $\si \in S_n$ such that $g(e_i)=\la e_{\si (i)} \la^{-1}$ for every $i$ (see loc. cit (i)). 
Since $g(e_i\La_n) = \la (e_{\si (i)}\La_n ) \la^{-1}$, we have $\dim_K  e_i \La_n = \dim_K e_{\si (i)}\La_n$ and therefore $\si (i)= i$ for every $i$. So now assume wlog. $g(e_i)=e_i$ for every $i$. We observe $e_{i+1} \La_n e_{i}\cong \Hom_{K[T]}(K[T]/(T^i), K[T]/(T^{i+1}))$ is generated by  $\alpha_i$ (corresponding to $K[T]/(T^i) \xrightarrow{T\cdot} K[T]/(T^{i+1})$) as a $e_{i+1} \La_n e_{i+1}\cong \End_{K[T]}(K[T]/(T^{i+1}))=K[T]/(T^{i+1})$-module.
Since $g(\alpha_i) \in e_{i+1}\La_n e_i$, there is a $\mu_{i+1} \in (e_{i+1}\La_n e_{i+1})^*$ such that $g(\alpha_i)= \mu_{i+1} \alpha_i$ for every $i$. Inductively, we can find $\la_i \in (e_i \La_n e_i)^*$ such that $\la_1=e_1$ and $ g(\alpha_i)\la_i=\la_{i+1}\alpha_i $ for every $i$. Define $\la :=\sum_{i=1}^n \la_i \in \La_n^*$, then $\la e_i \la^{-1}=e_i$ and $\la g (\alpha_j) \la^{-1} =\alpha_j$. \\
2. Secondly, we construct an isomorphism $\phi \colon \Aut_{K-alg}(K[T]/(T^n))\to H$ of groups. Let $g\in  \Aut_{K-alg}(K[T]/(T^n))$, then there exists $s_1\in K^*$, $s_j \in K$, $2\leq j\leq n-1$, such that $g(T)= \sum_{j=1}^{n-1}s_j T^j$. Then define $\phi (g)\in H$ via $\phi (g) (\beta_i):= \sum_{j=1}^{n-1} s_j (\beta_i \alpha_i)^{j-1} \beta_i $, it is easily be checked to define an injective group homomorphism. To see it is surjective, take $h \in H$, by the relation $\alpha_{i+1} h(\beta_{i+1})= h(\beta_i) \alpha_i$, it is easily checked that there is $s_1 \in K^*$, $s_j \in K$, $1 \leq j\leq n-1$ such that $h(\beta_i)= \sum_{j=1}^{n-1} s_j (\beta_i \alpha_i)^{j-1} \beta_i$, this clearly implies $h \in \Bild \phi$. \\
3. We want to see $\Aut_{K-alg}(K[T]/(T^n))\cong H \ltimes \Inn (\La_n)$, it is enough to show $\Inn (\La_n ) \cap H=\{\id \}$ since $\Inn (\La_n)$ is a normal subgroup, we have that $H$ operates by conjugation on it and the multiplication map $H\ltimes \rm{Inn}(\La_n) \to \Aut_{K-alg}(K[T]/(T^n)), (h,n)\mapsto hn$ fulfills: If its kernel ${\Inn} (\La_n ) \cap H=\{\id\}$, then it is surjective since then the image is a subgroup which contains ${\Inn} (\La_n ) \cup H$. \\
Take $g\in {\Inn} (\La_n ) \cap H$, so there exists $\la\in \La_n^*$ such that $g(x)=\la x \la^{-1}$ for all $x \in \La_n$. The equality $\la e_i = e_i \la =: \la_i \in (e_i \La e_i)^*$ implies $\la = \sum_{i=1}^n \la_i$. By $\la_i \alpha_i = \alpha_i \la_{i+1}$ for every $i$, one can check $\la = c_0 + c_1 \gamma + c_2 \gamma^2 +\cdots +c_{n-1}\gamma^{n-1}$ where $\gamma = \sum_{j=1}^{n-1} \alpha_j \beta_j$ and $c_i \in K$. It is easily checked that $\gamma$ is central and therefore $\la$ is in the center of $\La_n$ implying $g=\id$. 

\end{proof}

\begin{pro} \label{derived-Pic}
The following map 
\[
\begin{aligned} \rm{Out} (\La_n)\times \mathbb{Z} \times B_n &\to \rm{DPic}( \La_n)  \\
(\phi , m , x) &\mapsto ({}_{\phi^{-1}}\La_n \otimes_{\La_n}^{L}T_x)[m]
\end{aligned}
\]
is an isomorphism of groups.
\end{pro}

To make the proof easier to understand, we collect some easy observations on bimodules. 
\begin{lem} Let $\La$ be a finite-dimensional algebra and $\phi \in \Out (\La )$. Here we set $\otimes:= \otimes_{\La}$, we have the following maps $f$ which are isomorphisms of $\La$-$\La$-bimodules:  
\begin{itemize}
\item[(0)] $f\colon {}_{\phi} \La_\La \otimes {}_{\phi'} \La_{\La} \to {}_{\phi' \phi}\La_{\La} $,   $x \otimes y \mapsto \phi' (x)y$.
    \item[(1)] $f\colon {}_{\phi} \La_\La \otimes {}_{\phi^{-1}} \La_{\La} \to {}_{\La}\La_{\La} $,   $x \otimes y \mapsto \phi^{-1} (x)y$.
    \item[(2)] $f\colon {}_{\phi} \La_{\La} \to {}_{\La} \La_{\phi^{-1}}$, $x \mapsto \phi^{-1} (x)$.
    \item[(3)] Assume $I\subset \La$ is a two-sided ideal and $\phi (I) = I$ and $f\colon {}_{\La}I_{\La} \to {}_{\phi} I_{\phi}  $, $x \mapsto \phi (x)$ 
\end{itemize}
\end{lem}

\begin{proof} This is clear. 
\end{proof}
Now, recall that every $\phi \in H =\Out (\La_n )$ fulfills $\phi (e_i) =e_i$ for all $i$, this implies that every ideal $I_i= \La_n (1-e_i) \La_n$ fulfills $\phi (I_i) =I_i$. So clearly also all products and inverses of these ideals fulfill this. 
Also, we recall from \cite[section 3]{Y} that the following defines an injective group homomorphism for a finite-dimensional algebra $\La$ 
\[\Out (\La) \to \rm{Pic}(\La ) \subset \DPic (\La), \quad \phi \mapsto {}_{\phi^{-1}} \La_{\La}.
\]
where ${\rm Pic} (\La)$ is the group with objects isomorphism classes of invertible $\La$-bimodules with multiplication $\otimes=\otimes_{\La}$. 

\begin{lem} For every $\phi \in \Out (\La_n), x \in B_n$ 
we have ${}_{\phi}(\mathbb{T}_x)_{\phi}\cong\mathbb{T}_x $ in $D^b(\La_n^e\!-\!\modu )$. 
\end{lem}

\begin{proof}
We assume that $f=\phi \colon \mathbb{T}_x \to {}_{\phi} (\mathbb{T}_{x}) _{\phi} $ is an isomorphism of complexes of bimodules. Using the previous lemma 
\[ 
\begin{aligned}
{}_{\phi} (\mathbb{T}_{s_ix}) _{\phi} &\cong {}_{\phi} (\mathbb{T}_{s_i}) _{\La}\otimes^L {}_{\La} (\mathbb{T}_{x})_{\phi} \\
&\cong {}_{\phi} (\mathbb{T}_{s_i})_{\La}\otimes^L 
{}_{\La} \La_{\phi } \otimes^L  {}_{\phi } \La_{\La} \otimes^L 
{}_{\La} (\mathbb{T}_{x})_{\phi}\\
&\cong {}_{\phi} (\mathbb{T}_{s_i})_{\phi}\otimes^L {}_{\phi} (\mathbb{T}_{x})_{\phi} \cong \mathbb{T}_{s_i} \otimes^L \mathbb{T}_{x} \cong \mathbb{T}_{s_ix}
\end{aligned}
\]
Similarly, one shows ${}_{\phi}(\mathbb{T}_{s_i^{-1}x})_{\phi} \cong \mathbb{T}_{s_i^{-1}x}$. 
\end{proof}

Now, we can prove Proposition \ref{derived-Pic}, the proof is very similar to \cite[Theorem 4.4]{Mi17}.
\begin{proof}
1. First we show that $(\phi , m, x) \mapsto {}_{\phi^{-1}} \La_{\La} \otimes^L \mathbb{T}_x [m]$ is a group homomorphism. We have for $\phi, \phi' \in \Out (\La_n), x,x' \in B_n$
\[
\begin{aligned}
{}_{(\phi' \phi)^{-1}}\La\otimes^L \mathbb{T}_{x'x} &\cong {}_{\phi'^{-1}}\La \otimes^L {}_{\phi^{-1} } \La \otimes^L \mathbb{T}_{x'} \otimes^L \mathbb{T}_x \\
& \cong 
{}_{\phi'^{-1}}\La \otimes^L \La_{\phi} \otimes^L {}_{\phi}(\mathbb{T}_{x'})_{\phi} \otimes^L \mathbb{T}_x
\\
& \cong 
{}_{\phi'^{-1}}\La \otimes^L (\La_{\phi} \otimes^L {}_{\phi}\La) \otimes^L\mathbb{T}_{x'}\otimes^L \La_{\phi} \otimes^L \mathbb{T}_x \\
& \cong 
({}_{\phi'^{-1}}\La  \otimes^L\mathbb{T}_{x'})\otimes^L ({}_{\phi^{-1}}\La \otimes^L \mathbb{T}_x) 
\end{aligned}
\]
This shows that it is a group homomorphism.\\
2. Now, we show, the map is injective and surjective. 
Assume ${}_{\phi} (\La_n)_{\La_n} \otimes^L \mathbb{T}_x [m]=\La_n$ in $\DPic (\La_n)$. We apply the forgetful map and obtain $T_x[m]=(\La_n)_{\La_n}$ in $\tilt \La_n$, since the map $\rho$ is bijective, we have $m=0$ and $x=1\in B_n$. This implies ${}_{\phi}(\La_n)_{\La_n} \cong \La_n$ and therefore $\phi \in \Inn (\La_n )\cap \Out (\La_n)=\{\id \}$. This proves that the map is injective. Let $\mathbb{T} \in \DPic (\La_n )$. Then $\mathbb{T}_{\La}= T_x[m]\in \tilt \La_n$ for some $x \in B_n, m \in \mathbb{Z}$ since $\rho$ is bijective. By \cite[Prop. 2.3]{RZ03}, this implies that there is a $\phi \in \Out (\La_n)$ such that $\mathbb{T} = {}_{\phi} \La_n \otimes^L \mathbb{T}_x[m]$ in $\DPic (\La_n )$. This shows that the map is surjective. 
\end{proof}

Now for a finite-dimensional algebra $\La$, we use the following notation:\\
$2\!-\!\silt \La :=[\La[1], \La]$, $2\!-\!\tilt \La :=[\La [1], \La] \cap \tilt \La$ and $s\tau \!-\!\tilt \La$ denotes the set of isomorphism classes of basic support $\tau$-tilting $\La$-modules, cf. \cite{AIR} for the precise definition. We have the following corollary using the main result of \cite{IZ}. 
\begin{cor} \label{two-term}
There are bijections  
\[ S_n \sqcup S_n \longleftrightarrow \{ T \in s\tau\!-\!\tilt \La_n \mid Q_n \in \add (T) \text{ or } \Hom(Q_n, T)=0\} \longleftrightarrow 2\!-\!\tilt \La_n .\]
In particular, the number of $2$-term tilting complexes for $\La_n$ is $2n!$. \end{cor}

\begin{proof}
By Theorem \ref{main-thm}, we have $2\!-\!\tilt \La_n = 2\!-\!\silt_{Q_n} \La_n \sqcup 2\!-\!\silt_{Q_n[1]} \La_n $, so under  the Adachi-Iyama-Reiten bijection \cite[Theorem 3.2]{AIR}, this corresponds to the following set 
$\{ T \in s\tau\!-\!\tilt \La_n \mid Q_n \in \add (T) \text{ or } \Hom(Q_n, T)=0\}$. 
The result \cite[Theorem 1.1, (3)]{IZ} gives bijections $S_n \longleftrightarrow  \{ T \in s\tau\!-\!\tilt \La_n \mid Q_n \in \add (T) \}\longleftrightarrow  2\!-\!\silt_{Q_n} \La_n $, where the middle is the set of tilting modules of $\pdim \leq 1$. Using dagger duality \cite[Theorem 2.14]{AIR} one easily gets a bijection $ \{ T \in s\tau\!-\!\tilt \La_n \mid  \Hom(Q_n, T)=0\} \longleftrightarrow \{ T \in s\tau\!-\!\tilt \La_n \mid Q_n \in \add (T)\}$. By composing this with the first IZ-bijection from before, this completes the proof. 
\end{proof}



\section{Example}
We are looking at $\La_2$. Let $Q=Q_2=\begin{smallmatrix} 2\\1\\2 \end{smallmatrix}$ and $\beta \colon Q \to Q$ be the non-zero morphism factoring  over $S_2$. We define the following indecomposable objects in $\Ka^b(\La_2\!-\!\proj )$ 
for $n \geq 0$
\[
X_n = \underbrace{Q \xrightarrow{\beta} Q \xrightarrow{\beta} \cdots \xrightarrow{\beta} Q}_{n \text{ copies of  }Q} \to Q_1 \quad \quad 
Y_n = Q_1 \to \underbrace{Q \xrightarrow{\beta} Q \xrightarrow{\beta} \cdots \xrightarrow{\beta} Q}_{n \text{ copies of  }Q}
\]
where in $X_n, Y_n$ the module $Q_1$ sits in degree $0$. The mutation component of $\silt \Ka^b(\La_2\!-\!\proj )$ containing $\La_2$ is the following

\[
\xymatrix{
\vdots\ar[d]&&&&\\
X_2[\sm 2]\oplus Q\ar[d]\ar[r]&X_2[ \sm 2]\oplus X_3[ \sm 2]\ar@<.3em>[r]&X_2[\sm 2]\oplus X_3[\sm 1]\ar@<.3em>[r]\ar@<.3em>[l]|(0.48){[1]}
&X_2[\sm 2]\oplus X_{3}\ar@<.3em>[r]\ar@<.3em>[l]|(0.45){[1]}&\cdots\ar@<.3em>[l]|(0.35){[1]}\\
X_1[\sm 1]\oplus Q\ar[d]\ar[r]&X_1[ \sm 1]\oplus X_2[ \sm 1]\ar[lu]|{[1]}\ar@<.3em>[r]&X_1[ \sm 1]\oplus X_2\ar@<.3em>[r]\ar@<.3em>[l]|(0.45){[1]}
&X_1[ \sm 1]\oplus X_{2}[1]\ar@<.3em>[r]\ar@<.3em>[l]|{[1]}&\cdots\ar@<.3em>[l]|(0.28){[1]}\\
X_0\oplus Q\ar[d]\ar[r]&X_0\oplus X_1\ar[lu]|{[1]}\ar@<.3em>[r]&X_0\oplus X_1[1]\ar@<.3em>[r]\ar@<.3em>[l]|{[1]}
&X_0\oplus X_{1}[2]\ar@<.3em>[r]\ar@<.3em>[l]|{[1]}&\cdots\ar@<.3em>[l]|(0.4){[1]}\\
Y_1[1]\oplus Q\ar[d]\ar[r]&Y_1[1]\oplus X_0[1]\ar[lu]|{[1]}\ar@<.3em>[r]&Y_1[1]\oplus X_0[2]\ar@<.3em>[r]\ar@<.3em>[l]|{[1]}
&Y_1[1]\oplus X_{0}[3]\ar@<.3em>[r]\ar@<.3em>[l]|{[1]}&\cdots\ar@<.3em>[l]|(0.3){[1]}\\
Y_2[2]\oplus Q\ar[d]\ar[r]&Y_2[2]\oplus Y_1[2]\ar[lu]|{[1]}\ar@<.3em>[r]&Y_2[2]\oplus Y_1[3]\ar@<.3em>[r]\ar@<.3em>[l]|{[1]}
&Y_2[2]\oplus Y_{1}[4]\ar@<.3em>[r]\ar@<.3em>[l]|{[1]}&\cdots\ar@<.3em>[l]|(0.34){[1]}\\
\vdots&&&&
}
\]

\section*{Acknowledgement}
The author is supported by the Alexander von Humboldt Foundation in the framework of an Alexander von Humboldt Professorship endowed by the Federal Ministry of Education and Research. 
The author would like to thank William Crawley-Boevey, Jan Geuenich and Yuta Kimura for many helpful discussions.

\bibliographystyle{amsalpha}
\bibliography{main}
\end{document}